\documentclass[reqno,11pt,a4paper,final]{amsart}
\usepackage[a4paper,left=35mm,right=35mm,top=35mm,bottom=35mm,marginpar=25mm]{geometry} 
\usepackage{amsmath}
\usepackage{amssymb}
\usepackage{amsthm}
\usepackage{amscd}
\usepackage[ansinew]{inputenc}
\usepackage{cite}
\usepackage{bbm}
\usepackage{xcolor}
\usepackage[english=american]{csquotes}
\usepackage[final]{graphicx}
\usepackage{hyperref}
\usepackage{calc}
\usepackage{bm}
\usepackage{enumerate}

\graphicspath{{Pictures/}}

\numberwithin{equation}{section}

\newtheoremstyle{thmlemcorr}{10pt}{10pt}{\itshape}{}{\bfseries}{.}{10pt}{{\thmname{#1}\thmnumber{ #2}\thmnote{ (#3)}}}
\newtheoremstyle{thmlemcorr*}{10pt}{10pt}{\itshape}{}{\bfseries}{.}\newline{{\thmname{#1}\thmnumber{ #2}\thmnote{ (#3)}}}
\newtheoremstyle{defi}{10pt}{10pt}{\itshape}{}{\bfseries}{.}{10pt}{{\thmname{#1}\thmnumber{ #2}\thmnote{ (#3)}}}
\newtheoremstyle{remexample}{10pt}{10pt}{}{}{\bfseries}{.}{10pt}{{\thmname{#1}\thmnumber{ #2}\thmnote{ (#3)}}}
\newtheoremstyle{ass}{10pt}{10pt}{}{}{\bfseries}{.}{10pt}{{\thmname{#1}\thmnumber{ A#2}\thmnote{ (#3)}}}

\theoremstyle{thmlemcorr}

\newtheorem{theorem}{Theorem}
\numberwithin{theorem}{section}
\newtheorem{lemma}[theorem]{Lemma}
\newtheorem{corollary}[theorem]{Corollary}

\theoremstyle{thmlemcorr*}
\newtheorem{theorem*}{Theorem}
\newtheorem{lemma*}[theorem]{Lemma}
\newtheorem{corollary*}[theorem]{Corollary}
\newtheorem{proposition*}[theorem]{Proposition}
\newtheorem{problem*}[theorem]{Problem}
\newtheorem{conjecture*}[theorem]{Conjecture}
\newtheorem{question*}[theorem]{Question}

\theoremstyle{defi}

\theoremstyle{remexample}
\newtheorem{remark}[theorem]{Remark}

\theoremstyle{ass}

\newcommand{\Crm}{\mathrm{C}}

\newcommand{\Hrm}{\mathrm{H}}

\newcommand{\Lrm}{\mathrm{L}}

\newcommand{\Wrm}{\mathrm{W}}

\newcommand{\Acal}{\mathcal{A}}

\newcommand{\Ecal}{\mathcal{E}}
\newcommand{\Fcal}{\mathcal{F}}

\newcommand{\Scal}{\mathcal{S}}

\newcommand{\Abb}{\mathbb{A}}

\DeclareMathOperator{\id}{id}

\DeclareMathOperator{\diverg}{div}
\DeclareMathOperator{\curl}{curl}
\DeclareMathOperator{\dist}{dist}

\DeclareMathOperator{\spn}{span}

\DeclareMathOperator{\supp}{supp}

\newcommand{\ii}{\mathrm{i}}

\newcommand{\setb}[2]{\bigl\{\, #1 \ \ \textup{\textbf{:}}\ \ #2 \,\bigr\}}

\newcommand{\setBB}[2]{\biggl\{\, #1 \ \ \textup{\textbf{:}}\ \ #2 \,\biggr\}}

\newcommand{\dprb}[1]{\bigl\langle #1 \bigr\rangle}

\newcommand{\dd}{\;\mathrm{d}}

\newcommand{\N}{\mathbb{N}}
\newcommand{\R}{\mathbb{R}}
\newcommand{\C}{\mathbb{C}}

\newcommand{\loc}{\mathrm{loc}}
\newcommand{\sym}{\mathrm{sym}}

\newcommand{\ONE}{\mathbbm{1}}

\newcommand{\toLOne}{\overset{\Lrm^1}{\to}}
\newcommand{\weakONE}{\overset{\Lrm^1}{\rightharpoonup}}

\newcommand{\eps}{\epsilon}

\newcommand{\Rdds}{(\R^{d}\otimes\R^d)_\sym}

\DeclareMathOperator{\rk}{rk}

\def\XXint#1#2#3{{\setbox0=\hbox{$#1{#2#3}{\int}$} 
\vcenter{\hbox{$#2#3$}}\kern-.5\wd0}}

\renewcommand{\eps}{\varepsilon}
\renewcommand{\phi}{\varphi}

\title[Two-state problem for general differential operators]{On the two-state problem for general differential operators}

 \dedicatory{To Carlo Sbordone, for his 70th birthday.}

\author[G.~De Philippis]{Guido De Philippis}
\address{\textit{G.~De Philippis:} SISSA, Via Bonomea 265, 34136 Trieste, Italy.}
\email{guido.dephilippis@sissa.it}

\author[L.~Palmieri]{Luca Palmieri}
\address{\textit{L.~Palmieri:} SISSA, Via Bonomea 265, 34136 Trieste, Italy.} 
\email{contact@lpalmieri.com}

\author[F.~Rindler]{Filip Rindler}
\address{\textit{F.~Rindler:} Mathematics Institute, University of Warwick, Coventry CV4 7AL, UK, and The Alan Turing Institute, British Library, 96 Euston Road, London NW1 2DB London, UK.}
\email{F.Rindler@warwick.ac.uk}

\begin{document}

\begin{abstract}
In this note we generalize the Ball-James rigidity theorem for gradient differential inclusions to the setting of a general linear differential constraint. In particular, we prove the rigidity for approximate solutions to the two-state inclusion with incompatible states for merely $\Lrm^1$-bounded sequences. In this way, our theorem can be seen as a result of compensated compactness in the linear-growth setting.
\end{abstract}

\maketitle

\section{Introduction}

In~\cite{BallJames87} Ball and James proved the following rigidity property for gradient differential inclusions:

\begin{theorem}[Ball--James 1987~\cite{BallJames87}]\label{thm:BJ} Let \(\Omega\subset \R^d\) be an open, bounded, and connected set and let $A,B \in \R^m \otimes \R^d$ be $(m \times d)$-matrices with \(\rk (A-B)\ge 2\). Then:
\begin{itemize}
\item[(A)] If  \(u\in \Wrm^{1,\infty} (\Omega;\R^m) \) satisfies the differential inclusion
\begin{equation} \label{eq:diffincl}
    Du(x) \in \{A,B\}  \quad\text{for a.e. \(x\in \Omega\),}
\end{equation}
then either \(Du\equiv A\) or \(Du\equiv B\). 
\item[(B)] Let \((u_j)\subset \Wrm^{1,\infty}(\Omega;\R^m)\)  be a uniformly norm-bounded sequence of   maps such that
\[
\dist (Du_j ,\{A,B\})\to 0 \quad \text{in measure}.
\]
Then, up to extracting a subsequence, either
\[\qquad\qquad
\int_{\Omega} |Du_{j}(x) - A| \dd x \to 0 \quad \text{or} \quad  \int_{\Omega} |Du_{j}(x) - B| \dd x \to 0\qquad\textrm{as \(j\to \infty\).}
\]
\end{itemize}
\end{theorem}

The first part is known as {\em rigidity for exact solutions} of the differential inclusion~\eqref{eq:diffincl}, while the second one concerns {\em rigidity for approximate solutions}; see~\cite[Chapter 2]{Muller99} and~\cite[Chapter~8]{RindlerBOOK} for a discussion. Note also that, by a simple laminate construction, if \(\rk (A-B)=1\) then there exists a non-affine \(u\in \Wrm^{1,\infty}(\Omega;\R^m)\) such that \(Du(x)\in \{A,B\}\) for almost all \(x\in \Omega\). Moreover, laminates are the only possible solutions of this inclusion.

By recalling that on a simply connected domain \(v=Du\) if and only if \(\curl v=0\), the above results can be  summarized as follows:


\emph{ The inclusion 
\[
\begin{cases}
\curl v=0,\\
v\in \{A,B\}
\end{cases}
\]
is rigid for both exact and approximate solutions (in $\Lrm^\infty(\Omega;\R^{m \times d})$) if and only if \(\rk (A-B)\ge 2\).
}


In view of the above discussion it is natural to ask what happens when we replace the \(\curl\) with a general differential operator acting on vector-valued  functions  \(v\in \Crm^\infty(\Omega;\mathbb R^\ell)\), \(\Omega\subset \R^d\), namely
\begin{equation}\label{e:defA}
\Acal v:=\sum_{|\alpha|=k} A_\alpha \partial^\alpha v,
\end{equation}
where the sum is over all multi-indices \(\alpha \in (\N \cup \{0\})^d\), \(A_\alpha\in \R^n\otimes \R^\ell\) are (constant)  matrices (note that then the equation \(\Acal v=0\) is actually a {\em system of equations}).

Since the seminal work of Murat and Tartar~\cite{Tartar79,Tartar83,Murat78,Murat79} it is well understood that the role of rank-one matrices for gradient inclusions should be played by the {\em wave cone} associated to \(\Acal\):
\[
\Lambda_{\Acal}:=\bigcup_{|\xi|=1} \ker \Abb(\xi),\qquad 
\Abb(\xi):=(2\pi \ii)^{k}\sum_{|\alpha|=k} A_\alpha \xi^\alpha,
\]
where $\xi^\alpha=\xi_1^{\alpha_1}\cdots\xi_d^{\alpha_d}$. Indeed, one may easily check that if \(\lambda-\mu \in \ker \Abb(\xi)\) for some \(\xi\ne 0\) and \(h:\R\to \{0,1\}\) is measurable, then the function 
\begin{equation}\label{e:sol}
v(x):=\lambda h(x\cdot \xi)+\mu(1-h(x\cdot \xi) )
\end{equation}
is a solution to the differential inclusion 
\[
\begin{cases}
\Acal v=0  \quad\text{in the sense of distributions},\\
v\in \{\lambda,\mu\} \quad\text{a.e.}
\end{cases}
\]

Our main result generalizes the Ball--James Theorem~\ref{thm:BJ} to general operators $\Acal$ as above.

\begin{theorem}\label{thm:main}
Let \(\Omega\subset \R^d\) be an open, bounded, and connected set, and let \(\Acal\) be as in~\eqref{e:defA}. Suppose that $\lambda, \mu \in \R^\ell$ with $\lambda-\mu \notin \Lambda_\Acal$. Then:
\begin{itemize}
\item[(A)] If  \(v\in \Lrm^{\infty} (\Omega;\R^\ell) \) is such that 
\[
\Acal v=0\quad\textrm{in the sense of distributions}
\]
and 
\[
  v(x) \in \{\lambda,\mu\}  \quad\text{for a.e. \(x\in \Omega\),}
\]
then either \(v\equiv \lambda\) or \(v\equiv \mu\). 
\item[(B)] Let \((v_j)\subset \Lrm^1(\Omega;\R^\ell)\)  be a uniformly norm-bounded sequence  of maps such that 
\[
\Acal v_j=0\quad\textrm{in the sense of distributions}.
\]
Assume that
\[
\lim_{j\to \infty} \int_{\Omega} \dist (v_j(x),\{\lambda,\mu\}) \dd x =0.
\]
 Then, up to extracting a subsequence, either
\[\qquad\qquad
\int_{\Omega} |v_{j}(x) - \lambda| \dd x \to 0 \quad \textrm{or} \quad  \int_{\Omega} |v_{j}(x) - \mu| \dd x \to 0\qquad\textrm{as \(j\to \infty\).}
\]
\end{itemize}
\end{theorem}

\begin{remark}
As mentioned above, if \(\lambda-\mu\in \Lambda_\Acal\), then it is always possible to find a non-trivial solution of the exact differential inclusion. However, for operators of order greater or equal than two, not all the solutions are given by~\eqref{e:sol} and more general structures can arise. Classifying all of them would be an interesting problem.
\end{remark}



For particular choices of \(\Acal\), some instances of Theorem~\ref{thm:main} were already known: as mentioned before, if \(\Acal=\curl\), then it essentially reduces to the Ball--James Theorem~\eqref{thm:BJ} up to the  improvement on the summability assumption on the sequence \((v_j)_j=(Du_j)_j\) (which could also have been obtained in this particular case by the combination of Zhang's Lemma~\cite{Zhang92} and the simple Lemma~\ref{lm:mis} below). In the case when  \(\Acal=\diverg\), the above result was obtained by Garroni and Nesi~\cite{GarroniNesi04}. In the same paper the result is also generalized to  some other first-order operators. Other types of first-order operators  (essentially  combination of the \(\diverg\) and \(\curl\)) have been treated by Barchiesi in his master thesis~\cite{Bar03}. 

We refer to Remark~\ref{rmk:lo} for an extension of this theorem to non-homogeneous operators and non-zero (but asymptotically vanishing) right-hand sides in the constraint PDE.

Another interesting choice of operator is the so called ``\(\curl \curl\)''-operator, a second-order operator  whose kernel identifies  those symmetric matrix-valued fields \(S\colon \Omega \subset \R^d \to \Rdds\), $\Omega \subset \R^d$ a bounded simply-connected Lipschitz domain, for which there exists \(u \colon\Omega \to \R^d\)  such that
\[
S=\Ecal u:=\frac{Du+(Du)^T}{2}.
\]
For this operator the wave cone is given by the elementary symmetrized products,
\[
\Lambda_{\curl\curl}=\setBB{ a\odot \xi:=\frac{a\otimes \xi+\xi\otimes a}{2}}{ a,\xi \in \R^d},
\]
see for instance~\cite[Example 3.10(e)]{FonsecaMuller99} and~\cite[Proof of Theorem 1.7]{DePhilippisRindler16}. In this case Theorem~\ref{thm:main} implies the following:

\begin{corollary}\label{cor:BD}
Let \((u_j)_j\in \Wrm^{1,1}(\Omega;\R^d)\) be a sequence of maps such that 
\[
\sup_{j \in \N}\|\Ecal u_j\|_{\Lrm^1}<\infty \qquad\text{and}\qquad  \dist (\Ecal u_j, \{A,B\})\toLOne 0
\]
for some \(A,B \in \Rdds\) with \(A-B\notin \Lambda_{\curl\curl}\). Then, up to extracting a subsequence, either
\[
\int_{\Omega} |\Ecal u_{j}(x)- A| \dd x \to 0 \quad\textrm{or}\quad  \int_{\Omega} |\Ecal u_{j}(x)- B| \dd x \to 0\qquad\textrm{as \(j\to \infty\).}
\]
\end{corollary}


Other, more complicated, differential inclusion for (bounded) symmetric gradients have been studied in~\cite{DolzmannMuller95b,CapellaOtto09,CapellaOtto12,Rindler11,Ruland16,Ruland16b} (but without the linear-growth assumption).

Despite the fact that many interesting cases of Theorem~\ref{thm:main} were already known, we nevertheless believe that the generality of Theorem~\ref{thm:main} is of independent interest:

First, it extends all the known results to general operators of any order.

Second, our result for the first time enables one to treat the case of maps that are merely bounded in \(\Lrm^1\). Hence, Theorem~\ref{thm:main} may be seen as a result of  \emph{compensated compactness theory in $\Lrm^1$}.

Third, the previously-known cases of Theorem~\ref{thm:main}  have been proved by \textit{ad hoc} techniques. On the contrary, our proof exploits in quite a clear way the heuristic  principle  that {\em when restricted to maps whose image is not in the wave cone, \(\Acal\) is elliptic} (which in this generality is too good to be true); see Section~\ref{s:2} for a more precise statement. Of course, this idea lies at the root of the Murat--Tartar compensated compactness theory, which also provides the conceptual framework of the present work. In the proof of Theorem~\ref{thm:main} we exploit some ideas which have been introduced in~\cite{DePhilippisRindler16}, where the first and third author proved the  natural  generalization of Alberti's rank-one theorem to the setting of \(\Acal\)-free measures.

We conclude this introduction by mentioning that one can also study more complicated differential inclusions. In the case of \(\Acal=\curl\), this has been a very active and fruitful area of research with several deep results  and nice  connections with other problems in mathematics.  Since here we cannot give a detailed account of the (extensive) literature,  we refer the reader to~\cite[Chapter~8]{RindlerBOOK} for a recent survey on the theory. 

It is worth noting that in the study of more general inclusions in the general setting of Theorem~\ref{thm:main}, the intrinsic geometry of the operator \(\Acal\) should play an important role. Indeed, for differential inclusions involving more than two states the type of results available seems to be  different depending on the particular choice of \(\Acal\). For instance, when \(\Acal=\curl\) it has been proved by \v{S}ver\'{a}k that  the three-state problem is  rigid  both for approximate and exact solutions~\cite{Sverak91b}. On the contrary, when \(\Acal=\diverg\), in~\cite[Section 4]{GarroniNesi04} it is shown that the three state-problem is not rigid for approximate solutions, while Palombaro and Ponsiglione~\cite{PalombaroPonsiglione04} showed that it is rigid for exact solutions; see also~\cite{PalombaroSmyshlyaev09}.


 \subsection*{Acknowledgements}

G.~D.~P.\ is supported by the MIUR SIR-grant ``Geometric Variational Problems" (RBSI14RVEZ). This project has received funding from the European Research Council (ERC) under the European Union's Horizon 2020 research and innovation programme (grant agreement 757254) for the project ``SINGULARITY''.  F.~R.\ also acknowledges the support from an EPSRC Research Fellowship on ``Singularities in Nonlinear PDEs'' (EP/L018934/1).

\section{Proof of Theorem~\ref{thm:main}}\label{s:2}

In this Section we prove Theorem~\ref{thm:main}. As mentioned in the introduction, the proof is based on the observation (which lies at the very definition of the wave cone) that, if \(\lambda \notin \Lambda_\Acal\),  then the operator \(\Acal\) restricted to functions whose image lies in \(\spn\{\lambda\}\) is {\em elliptic}  in the sense that is Fourier symbol does not vanish.  Exploiting this elliptic regularization together with some classical estimates in harmonic analysis in the spirit of~\cite{DePhilippisRindler16} allows us to deduce the desired rigidity.

We start with an elementary  lemma, which relates the \(\Lrm^1\)-convergence to zero of \(\dist(v_j,\{\lambda,\mu\})\) to the equi-integrability of the sequence \((v_j)_j\).
Recall that a sequence of measurable functions \(f_j \colon \Omega \to \R\) ($j \in \N$) is  said to converge  to \(0\) in measure if 
\[
\int_{\Omega} \min\{|f_j|,1\}\dd x\to 0\qquad \textrm{as \(j\to \infty\).}
\]

\begin{lemma}\label{lm:mis}
Let \(\Omega\subset \R^d\) be an open and bounded set, let \((v_j)_j\subset \Lrm^1(\Omega;\R^\ell)\) and let \(\lambda,\mu \in \R^\ell\). Then the following are equivalent:
\begin{itemize}
\item[(i)] \(\dist(v_j,\{\lambda,\mu\})\toLOne 0\);
\item[(ii)] the sequence \((v_j)_j\) is equi-integrable, i.e.,
\[
\lim_{t\to \infty} \sup_{j \in \N} \int_{\Omega\cap \{|v_j|> t\} }|v_j| \dd x =0, 
\]
and \(\dist(v_j,\{\lambda,\mu\})\to 0\) in measure;
\item[(iii)] the sequence \((v_j)_j\) is equi-integrable and \(\dist(v_j,\{\lambda,\mu\})\toLOne 0\).
\end{itemize}
\end{lemma}

\begin{proof}
For \(t\ge 2\max \{|\lambda|,|\mu|\}\) we have
\begin{equation*}\label{e:s}
\big\{\dist(v,\{\lambda, \mu\})>2t\big\}\subset\big \{|v|>t\big\}\subset \big\{\dist(v,\{\lambda, \mu\})>t/2\big\}.
\end{equation*}
Hence, \((v_j)_j\) is equi-integrable if and only if \((\dist(v_j,\{\lambda, \mu\}))_j\) is equi-integrable. By Vitali's Theorem (which says that  sequence of equi-integrable functions converges to \(0\)  in measure if and only if it converges to \(0\) in \(\Lrm^1\)) together with the fact that $\Lrm^1$-converging sequences are equi-integrable, we immediately get that (i), (ii), and (iii) are equivalent.
\end{proof}


\begin{proof}[Proof of Theorem~\ref{thm:main}~(A)]




Shifting \(\tilde{v}(x) := v(x)-\mu\),  we can assume that \(\mu=0\) and \(v=\lambda \ONE_E\) with \(\lambda \notin \Lambda_{\Acal}\)  and 
\[
E=\setb{x\in \Omega}{ v(x)=\lambda}.
\]
The goal is to show that either \(E=\Omega\) or \(E=\emptyset\). We will prove that \(\ONE_E\in \Crm^\infty(\Omega)\), which immediately implies the conclusion since \(\Omega\) is connected. In order to do so, let \(\phi\in \Crm^\infty_c(\Omega)\) and let us consider the function \(w:=\phi  \ONE_E\in \Lrm^2(\R^d)\). Then,
\begin{equation}\label{e:key}
\Acal(\lambda w) =\phi \Acal (\lambda \ONE_E)+[\Acal,\phi](\lambda \ONE_E)=[\Acal,\phi](\psi \lambda \ONE_E),
\end{equation}
where \(\psi\in \Crm^\infty_c(\Omega)\) is identically \(1\) in \(\supp \phi\) and  for \(f\in\Lrm^2(\R^d) \) we have defined
\begin{align}
[\Acal,\phi](f) &:=\Acal( \phi f)-\phi \Acal (f)  \notag
\\
&\phantom{:}=\sum_{h=1}^{k-1}\sum_{|\beta|=h} B_{\beta}(x) \partial^\beta f  \label{e:comm}
\end{align}
with some coefficient matrices $B_\beta \in \Crm^\infty_{c}(\supp \phi;\R^n\otimes \R^\ell)$. The key point is that the commutator $[\Acal,\phi]$ is an operator of order \((k-1)\). More precisely, one easily checks that    
\[
[\Acal,\phi]: \Hrm^{s}(\R^d)\to \Hrm^{s-(k-1)}(\R^d)
\]
 for all \(s\in \R\), see for instance~\cite[Proof of Theorem 9.26]{Folland99book} or part~(B) of the proof below.  Here, \(\Hrm^s(\R^d)\) is the classical \(\Lrm^2\)-based  Sobolev space, that is,
\begin{equation}\label{e:sob}
f\in \Hrm^s(\R^d) \iff (\id - \Delta)^{s/2} f \in \Lrm^2(\R^d) \iff (1+4\pi^2|\xi|^2)^{s/2} \hat{f}(\xi) \in \Lrm^2(\R^d),
\end{equation}
where we denote the Fourier transform of $f \in \Lrm^2(\R^d)$ by
\[
\hat f(\xi)=\Fcal [f](\xi):=\int f(x) e^{-2\pi \ii x\cdot \xi} \dd x,
\]
which also lies $\Lrm^2(\R^d)$ by the Plancherel theorem. For general tempered distributions $f$ we define $\hat{f}$ in the distributional sense, see~\cite{Grafakos14book1} for details. 

Taking the Fourier transform of~\eqref{e:key}, we thus obtain
\[
\Abb(\xi) \lambda\, \hat w(\xi)=\Fcal \bigl[ [\Acal,\phi](\psi \lambda \ONE_E) \bigr](\xi),  \qquad \xi \in \R^d.
\]
Scalar multiplying the above equation with \(\overline{\Abb(\xi) \lambda}\in \C^n\),
\begin{equation}\label{e:key2}
\big(1+|\Abb(\xi)\lambda|^2\big)\hat w(\xi)=\overline{\Abb(\xi) \lambda}\cdot \Fcal \bigl[ [\Acal,\phi](\psi \lambda \ONE_E) \bigr](\xi)+\hat w (\xi):=\hat R (\xi).
\end{equation}
Since by assumption \(\lambda\notin \Lambda_\Acal\), there exists \(c>0\) such that  \(|\Abb(\xi)\lambda|\ge c|\xi|^{k}\) for all $\xi \in \R^d$. Hence, by the very definition of Sobolev spaces~\eqref{e:sob}, the operator \(T\) defined by 
\[
T(f) := \Fcal^{-1}\Bigg(\frac{\hat f(\xi)}{1+|\Abb(\xi)\lambda|^2}\Bigg)
\]
maps \(\Hrm^{s}(\R^d)\) into \(\Hrm^{s+2k}(\R^d)\) for all \(s\in \R\). Since, by our initial assumption, \( \ONE_E \in \Lrm^2(\R^d)=\Hrm^{0} (\R^d)\), we have that the right-hand side \(\hat R\) of~\eqref{e:key2} belongs to \( \Fcal(\Hrm^{1-2k}(\R^d))\). Thus,
\[
w=\phi \ONE_E=T(\Fcal^{-1}(\hat R)) \in \Hrm^1(\R^d) \qquad\textrm{for all \(\phi \in \Crm^1_c(\Omega)\)}.
\]
 This in turn  yields \(\hat R\in \Fcal(\Hrm^{2-2k}(\R^d))\), which implies that \(w=\phi \ONE_E\in \Hrm^2(\R^d)\) for all \(\phi \in \Crm^1_c(\Omega)\). Iterating, we obtain the desired regularity for \(\ONE_E\) and we conclude the proof (in fact, already the first step implies that \(\ONE_E\in \Hrm^1_{\rm loc}(\Omega)\) and this is enough to conclude since there are no non-constant characteristic functions in \(\Hrm^1_{\rm loc}(\Omega)\)).
\end{proof}


\begin{proof}[Proof of Theorem~\ref{thm:main}~(B)]
 As in the proof for~(A) we can assume that \(\mu=0\) and \(\lambda\notin \Lambda_\Acal\). By Lemma~\ref{lm:mis} the sequence \((v_j)\) is equi-integrable. Hence, by the Dunford--Pettis compactness theorem (see, for instance,~\cite[Theorem 1.38]{AmbrosioFuscoPallara00book}),  there exists a (non-relabeled) subsequence and a function \(v\in \Lrm^1(\Omega)\) such that 
\[
v_j \weakONE v.
\]
With
\[
  \eps_j :=\|\dist(v_{j}, \{\lambda,0\})\|_{\Lrm^1}^{1/2}\to 0
\]
let us define
\[
E_j:=\setb{x\in \Omega}{ |v_j(x)-\lambda| \le \eps_j}.
\]
We first show that
\begin{equation}\label{e:lim}
\|v_j-\lambda\ONE_{E_j}\|_{\Lrm^1} \to 0.
\end{equation}
Indeed, 
\[
A_j:= \setb{x\in \Omega}{ \dist(v_{j}(x), \{\lambda,0\})\le \eps_j}= E_j \cup \setb{x\in \Omega}{ |v_{j}(x)| \le \eps_j }
\]
and the union is disjoint provided \(\eps_j\) is small enough. Moreover,
\[
|\Omega\setminus A_j|\le \frac{1}{\eps_j} \int_{\Omega}  \dist(v_{j}, \{\lambda,0\}) \dd x\to 0.
\]
Hence,
\[
\begin{split}
\int_{\Omega} |v_j-\lambda\ONE_{E_j}| \dd x &= \int_{E_j }|v_j-\lambda\ONE_{E_j}| \dd x + \int_{\{|v_j|\le \eps_j\}}|v_j| \dd x  +  \int_{\Omega\setminus A_j }|v_j| \dd x 
\\
 &\le 2\eps_j|\Omega|+ \int_{\Omega\setminus A_j }|v_j| \dd x \to 0
\end{split}
\]
since the sequence \((v_j)\) is equi-integrable and  \(|\Omega\setminus A_j|\to 0\). This shows~\eqref{e:lim} and then also
\begin{equation}\label{e:lim2}
\lambda \ONE_{E_j} \weakONE v.
\end{equation}

We now claim:
\begin{equation}\label{e:claim}
\textrm{\emph{The sequence \( \ONE_{E_j} \) is (strongly) pre-compact in \(\Lrm^1(\Omega)\)}.}
\end{equation}
Before proving this claim, let us show  how this gives  the desired conclusion. Indeed,~\eqref{e:lim},~\eqref{e:lim2} and~\eqref{e:claim} imply that
\begin{equation}\label{e:pippo}
\|\lambda \ONE_{E_j}-v\|_{\Lrm^1} + \|v_j-v\|_{\Lrm^1}\to 0.
\end{equation}
In particular, \(v= \lambda \ONE_E\) for some set \(E\subset \Omega\). Since clearly \(\Acal v=0\), part (A) implies that either \(E=\Omega\) or \(E=\emptyset\). Together with~\eqref{e:pippo} this concludes the proof.

\medskip

We are thus left to show the claim~\eqref{e:claim}. By the equi-integrability, it is enough to show that \(w_j:=\phi\ONE_{E_j}\) is pre-compact for all \(\phi \in \Crm^\infty_c(\Omega)\). As in part (A) we have
\[
\Acal (\lambda w_j)=\Acal (\lambda w_j-\phi v_j)+\Acal (\phi v_j)= \Acal (\lambda w_j-\phi v_j)+[\Acal, \phi](\psi  v_j),
\] 
where \(\psi\in \Crm^\infty_c(\Omega) \) is identically \(1\) on \(\supp \phi\) and \([\Acal, \phi]\) is defined in~\eqref{e:comm}. Taking the Fourier transform of the above equation and multiplying by \(\overline{\Abb(\xi) \lambda}\), we get
\[
(1+|\Abb(\xi) \lambda|^2) \hat{w}_j(\xi)=\overline{\Abb(\xi) \lambda} \cdot \Abb(\xi) \hat{z}_j(\xi)+\overline{\Abb(\xi) \lambda} \cdot \Fcal \bigl[ [\Acal, \phi](u_j) \bigr](\xi)+\hat{w}_j(\xi),
\]
where we have set
\[
   z_j:= \lambda w_j-\phi v_j,  \qquad
   u_j:=\psi v_j.
\]
Hence,
\begin{equation}\label{e:dec}
w_j=T_1(z_j)+ T_{2}(u_j)+T_3(w_j),
\end{equation}
where \(z_j, u_j, w_j\in \Lrm^1(\supp \psi)\) and 
\[
\begin{split}
T_1(z) &:= \Fcal^{-1} \Bigg(\frac{\overline{\Abb(\xi) \lambda} \cdot \Abb(\xi) \hat z(\xi)}{1+|\Abb(\xi) \lambda|^2}\Bigg),
\\
T_2(u) &:= \Fcal^{-1} \Bigg(\frac{\overline{\Abb(\xi) \lambda} \cdot \Fcal\bigl[[\Acal, \phi](u)\bigr](\xi)}{1+|\Abb(\xi) \lambda|^2}\Bigg),
\\
T_3(w) &:= \Fcal^{-1} \Bigg(\frac{\hat w(\xi)}{1+|\Abb(\xi) \lambda|^2}\Bigg).
\end{split}
\]
We also have, by~\eqref{e:lim},
\begin{equation}\label{e:stime}
z_j = \phi (\lambda \ONE_{E_j}-v_j) \toLOne 0,  \qquad  \sup_{j \in \N}\int_{\R^d} |w_j|+|u_j|<\infty.
\end{equation}

We now claim that
\begin{equation} \label{e:claim2}
\left\{
\begin{aligned}
&\|T_1(z_j)\|_{\Lrm^{1,\infty}}:=\sup_{t>0} \, t|\{|T_1(z_j)|>t\}\to 0;\\
&T_1(z_j)\to 0 \quad\textrm{in the sense of distributions};
\\
&\text{\((T_2(u_j))_j\), \((T_3(w_j))_j\)  are pre-compact in \(\Lrm_{\rm loc}^1\)}.
\end{aligned}
\right.
\end{equation}
The above facts {\em together with  the positivity of \(w_j\)} imply~\eqref{e:claim}, see~\cite[Proof of Theorem 1.1]{DePhilippisRindler16} and also~\cite[Lemma 3.3]{DePDeRGhi2}, where the convergence lemma is explicitly stated.

The proof of~\eqref{e:claim2} closely follows the proof of~\cite[Theorem 1.1]{DePhilippisRindler16}. Indeed, by assumption, \(|\Abb(\xi) \lambda|\ge c |\xi|^k\) for some \(c>0\). Thus, the H\"ormander-Mihlin multiplier theorem, see~\cite[Theorem 5.2.7]{Grafakos14book1}, yields that $T_1$ is a bounded operator from $\Lrm^1(\R^d)$ to $\Lrm^{1,\infty}(\R^d)$ and in particular, via~\eqref{e:stime},
\[
\|T_1(z_j)\|_{\Lrm^{1,\infty}}\le C\|z_j\|_{\Lrm^1}\to 0.
\]
Moroever, for \(\zeta \in \Crm^\infty_c(\R^d)\),
\[
\dprb{ T_1(z_j),\zeta}=\dprb{ z_j,T^*_1(\zeta)} \to 0,
\]
where \(T^*_1:  \Crm^\infty_c(\R^d)\to \Scal(\R^d)\), the Schwartz space of test functions, is the adjoint of \(T_1\). 

We are thus left with the task to show that \((T_2(u_j))_j\) and \((T_3(w_j))_j\) are pre-compact in  \(\Lrm_{\rm loc}^1\). To this end note that,   by~\eqref{e:comm},  \(T_2(u_j)\)  can be  written as a finite  sum of terms of the form 
\begin{equation*}\label{fj}
f_j^\beta= Q\circ(\id-\Delta)^{-\frac{k}{2}} \circ P_{\beta} \circ (\id-\Delta)^{\frac{|\beta|-k}{2}} [u^\beta_j],
\end{equation*}
where \(\beta\in (\N \cup \{0\})^{d}\),  \(0\le |\beta|\le (k-1)\), \(\sup_{j \in \N}\|u^\beta_j\|_{\Lrm^1}\le C\),
\[
Q[u]=\Fcal^{-1} \Bigl[(1+|\Abb(\xi)\lambda |^2)^{-1}(1+4\pi^2|\xi|^2)^{k/2}\, \overline{\Abb(\xi)\lambda} \hat {u}(\xi)\Bigr] ,
\]
and \(P_\beta\) is the $k$'th-order pseudo-differential operator given by
\[
P_{\beta}[u](x)=\int \frac{(2\pi\ii)^{|\beta|}\xi^\beta}{(1+4\pi^2|\xi|^2)^{\frac{|\beta|-k}{2}}} \,B_{\beta}(x) \hat{u}(\xi) \, e^{2\pi\ii x \cdot \xi} \dd \xi.
\]
The composition \((\id-\Delta)^{-k/2} \circ P_{\beta}\) is a pseudo-differential operator of order \(0\), see~\cite[Theorem 2, Chapter VI]{Stein93book},  and thus bounded from  \(\Lrm^p(\R^d)\) to itself, see~\cite[Proposition 4 in Chapter VI]{Stein93book}. By  the   H\"ormander--Mihlin multiplier theorem,  aasslso \(Q\) is a bounded  operator  from  \(\Lrm^p(\R^d)\) to itself.  Since \(|\beta|\le k-1\), one easily checks that \((\id-\Delta)^{(|\beta|-k)/2}\) is compact from \(\Lrm^1(\supp \psi) \)  to \(\Lrm^p(\R^d)\) for \(1<p<p(d,k-|\beta|)\), see for instance~\cite[Lemma 2.1]{DePhilippisRindler16}. We thus infer that \((T_2(u_j))_{j}\) is pre-compact in \(\Lrm^1_\loc(\R^d)\). Since a similar (and actually easier) argument applies to \((T_3(w_j))_{j}\), this concludes  the proof of~\eqref{e:claim2} and thus of the theorem.
\end{proof}

\begin{remark}\label{rmk:lo} It is clear from the proof that Theorem \ref{thm:main} can be extended to include non-homogeneous operators and (asymptotically vanishing) right-hand sides. More precisely,  assume that 
\[
\begin{split}
\Acal=\sum_{|\alpha|\le k} A_\alpha \partial^\alpha={\sum_{|\alpha|=k} A_\alpha \partial^\alpha}+\sum_{|\alpha|\le (k-1)} A_\alpha \partial^\alpha
:=\Acal^k+\Acal^{<k}
\end{split}
\]
and that \((v_j)_j\subset \Lrm^1(\Omega,\R^\ell)\) satisfies
\[
\Acal v_j=r_j, \qquad \dist(v_j,\{\lambda, \mu\})\toLOne 0,
\]
where  \(r_j\) can be written as 
\begin{equation}\label{e:rem}
r_j=\sum_{|\beta|=k}\partial^{\beta} f^\beta_j\qquad\text{with}\qquad \|f_j^\beta\|_{\Lrm^1}\to 0.
\end{equation}
 Moreover, we assume
\[
\lambda-\mu\notin \Lambda_{\Acal^k}:=\bigcup_{|\xi|=1} \ker\Abb^k(\xi),\qquad \Abb^k(\xi):=(2\pi\ii)^k\sum_{|\alpha|=k}A_\alpha\xi^\alpha.
\]
Then, up to taking a subsequence, either \(\|v_j-\lambda\|_{\Lrm^1}\to 0\) or \(\|v_j-\mu\|_{\Lrm^1}\to 0\). 
Note in particular that \eqref{e:rem}   includes the case in which \(r_j\to 0\) strongly in \(W^{-k,p'}\) for \(p>1\).

With these assumptions, the proof of Theorem \ref{thm:main} can be repeated almost verbatim with \(\Acal\) replaced by \(\Acal^{k}\). The presence of the non-homogeneous terms \(\Acal^{<k}v_j\)  and  \(r_j\) can be dealt with by adding them to the right-hand side of~\eqref{e:dec}. Indeed, this yields additional terms of the form
\[
\begin{split}
T_{4}( v_j)&:=\Fcal^{-1}\Bigg( \frac{\overline{\Abb^k(\xi)\lambda}\,\Fcal\big[\phi\Acal^{<k} v_j\big](\xi)}{1+|\Abb^k(\xi) \lambda|^2}\Bigg),
\\
 T_5(f^\beta_j)&:=\Fcal^{-1}\Bigg( \frac{(2\pi \ii)^k\overline{\Abb^k(\xi)\lambda}\,\xi^{\beta}\hat{f}_j^\beta(\xi)}{1+|\Abb^k(\xi) \lambda|^2}\Bigg).
\end{split}
\]
By the H\"ormander--Mihlin theorem, \(T_{5}(f^\beta_j)\) converges to \(0\) in \(\Lrm^{1,\infty}\) and in the sense of distributions, while, by the same reasoning used for  \((T_2(u_j))_{j}\), one shows that \((T_4(u_j))_j\) is pre-compact in \(\Lrm^1_{\rm loc}\). Gathering together these facts, we may conclude the proof by the very same arguments as in Theorem \ref{thm:main}.

\end{remark}



\providecommand{\bysame}{\leavevmode\hbox to3em{\hrulefill}\thinspace}
\providecommand{\MR}{\relax\ifhmode\unskip\space\fi MR }
\providecommand{\MRhref}[2]{%
  \href{http://www.ams.org/mathscinet-getitem?mr=#1}{#2}
}
\providecommand{\href}[2]{#2}

\end{document}